\documentclass{article}
\usepackage{microtype}
\usepackage{graphicx}
\usepackage{subfigure}
\usepackage{booktabs}
\usepackage{hyperref}
\usepackage{amsmath, amssymb, amsthm}
\usepackage[fixamsmath]{mathtools}
\mathtoolsset{showmanualtags,mathic,centercolon}
\usepackage{amsfonts}
\usepackage{balance}
\graphicspath{{imgs/}}
\usepackage[accepted]{icml2018}

\icmltitlerunning{Alternating Randomized Block Coordinate Descent}

\newcommand{\vx}{\mathbf{x}}
\newcommand{\vxh}{\mathbf{\hat{x}}}
\newcommand{\vy}{\mathbf{y}}
\newcommand{\vz}{\mathbf{z}}
\newcommand{\vv}{\mathbf{v}}
\newcommand{\vu}{\mathbf{u}}
\newcommand{\vw}{\mathbf{w}}
\newcommand{\vb}{\mathbf{b}}
\newcommand{\zeros}{\mathbf{0}}
\newcommand{\mM}{\mathbf{M}}
\newcommand{\mB}{\mathbf{B}}
\newcommand{\mC}{\mathbf{C}}
\newcommand{\vr}{\mathbf{r}}
\newcommand{\innp}[1]{\left\langle #1 \right\rangle}

\newcommand*{\vsepfbox}[1]{%
  \begingroup
    \sbox0{\fbox{#1}}%
    \setlength{\fboxrule}{0pt}%
    \mbox{\kern-\fboxsep\fbox{\unhbox0}\kern-\fboxsep}%
  \endgroup
}
\newcommand{\littlesum}{\mathop{\textstyle\sum}}
\newcommand{\defeq}{\stackrel{\mathrm{\scriptscriptstyle def}}{=}}

\theoremstyle{plain} \numberwithin{equation}{section}
\newtheorem{theorem}{Theorem}[section]
\numberwithin{theorem}{section}

\newtheorem{lemma}[theorem]{Lemma}
\newtheorem{proposition}[theorem]{Proposition}
\newtheorem{observation}[theorem]{Observation}

\theoremstyle{definition}

\newtheorem{remark}[theorem]{Remark}
\DeclareMathOperator*{\argmin}{argmin}

\begin{document}

\twocolumn[
\icmltitle{Alternating Randomized Block Coordinate Descent}

% It is OKAY to include author information, even for blind
% submissions: the style file will automatically remove it for you
% unless you've provided the [accepted] option to the icml2018
% package.

% List of affiliations: The first argument should be a (short)
% identifier you will use later to specify author affiliations
% Academic affiliations should list Department, University, City, Region, Country
% Industry affiliations should list Company, City, Region, Country

% You can specify symbols, otherwise they are numbered in order.
% Ideally, you should not use this facility. Affiliations will be numbered
% in order of appearance and this is the preferred way.
\icmlsetsymbol{equal}{*}

\begin{icmlauthorlist}
\icmlauthor{Jelena Diakonikolas}{bu}
\icmlauthor{Lorenzo Orecchia}{bu}
\end{icmlauthorlist}

\icmlaffiliation{bu}{Department of Computer Science, Boston University, Boston, MA, USA}

\icmlcorrespondingauthor{Jelena Diakonikolas}{jelenad@bu.edu}
\icmlcorrespondingauthor{Lorenzo Orecchia}{orecchia@bu.edu}

% You may provide any keywords that you
% find helpful for describing your paper; these are used to populate
% the "keywords" metadata in the PDF but will not be shown in the document
\icmlkeywords{Alternating minimization, block coordinate descent, optimization}

\vskip 0.3in
]

% this must go after the closing bracket ] following \twocolumn[ ...

% This command actually creates the footnote in the first column
% listing the affiliations and the copyright notice.
% The command takes one argument, which is text to display at the start of the footnote.
% The \icmlEqualContribution command is standard text for equal contribution.
% Remove it (just {}) if you do not need this facility.

%\printAffiliationsAndNotice{}  % leave blank if no need to mention equal contribution
\printAffiliationsAndNotice{} % otherwise use the standard text.
\begin{abstract}
Block-coordinate descent algorithms and alternating minimization methods are fundamental optimization algorithms and an important primitive in large-scale optimization and machine learning. While various block-coordinate-descent-type methods have been studied extensively, only alternating minimization -- which applies to the setting of only two blocks -- is known to have convergence time that scales independently of the least smooth block. A natural question is then: is the setting of two blocks special? 

We show that the answer is ``no'' as long as the least smooth block can be optimized exactly -- an assumption that is also needed in the setting of alternating minimization. We do so by introducing a novel algorithm~\ref{eq:AR-BCD}, whose convergence time scales independently of the least smooth (possibly non-smooth) block. The basic algorithm generalizes both alternating minimization and randomized block coordinate (gradient) descent, and we also provide its accelerated version -- \ref{eq:AAR-BCD}. As a special case of~\ref{eq:AAR-BCD}, we obtain the first nontrivial accelerated alternating minimization algorithm.

%In this paper, we introduce a novel algorithm~\ref{eq:AR-BCD} for the optimization of smooth convex functions in the first-order-oracle model. Our algorithm generalizes both randomized block-coordinate and alternating minimization methods by allowing exact optimization over a single non-smooth block, and performing randomized block-coordinate gradient descent on the remaining blocks. We also provide an accelerated version of the algorithm, whose oracle complexity matches the best known bounds for block-coordinate descent, but is also able to accommodate exact optimization over a non-smooth block. Finally, we perform an experimental evaluation of our new algorithms and compare their performance with existing methods under different choices of smoothness and size for the different variable blocks.
\end{abstract}

\section{Introduction}\label{sec:intro}

First-order methods for minimizing smooth convex functions are a cornerstone of large-scale optimization and machine learning. Given the size and heterogeneity of the data in these applications, there is a particular interest in designing iterative methods that, at each iteration, only optimize over a subset of the decision variables~\cite{wright2015coordinate}.
% LO: CITATIONS?

This paper focuses on two classes of methods that constitute important instantiations of this idea. The first class is that of {\it block-coordinate descent methods}, i.e., methods that partition the set of variables into $n \geq 2$ blocks and perform a gradient descent step on a single block at every iteration, while leaving the remaining variable blocks fixed. A paradigmatic example of this approach is the randomized Kaczmarz algorithm of~\cite{SV} for linear systems and its generalization~\cite{nesterov2012efficiency}. 
The second class is that of {\it alternating minimization methods}, i.e., algorithms that partition the variable set into only $n=2$ blocks and alternate between {\it exactly optimizing} one block or the other at each iteration (see, e.g.,~\cite{beck2015convergence-AM} and references therein). 

Besides the computational advantage in only having to update a subset of variables at each iteration, methods in these two classes are also able to exploit better the structure of the problem, which, for instance, may be computationally expensive only in a small number of variables. 
To formalize this statement, assume that the set of variables is partitioned into $n \leq N$ mutually disjoint blocks, where the $i^{\mathrm{th}}$ block of variable $\vx$ is denoted by $\vx^{i}$, and the gradient corresponding to the $i^{\mathrm{th}}$ block is denoted by $\nabla_i f(\vx)$. Each block $i$ will be associated with a smoothness parameter $L_i,$ I.e., $\forall \vx, \vy \in \mathbb{R}^N$: 
\begin{equation}\label{eq:smoothness-grad}
\|\nabla_i f(\vx + I_N^i\vy) - \nabla_i f(\vx)\|_* \leq L_i \|\vy^i\|,
\end{equation}
where $I_N^i$ is a diagonal matrix whose diagonal entries equal one for coordinates from block $i$, and are zero otherwise. 

In this setting, the convergence time of standard randomized block-coordinate descent methods, such as those in~\cite{nesterov2012efficiency}, scales as $O\left(\frac{\sum_i L_i}{\epsilon}\right)$, where $\epsilon$ is the desired additive error. By contrast, when $n=2,$ the convergence time of the alternating minimization method~\cite{beck2015convergence-AM} scales as $O\left(\frac{L_{\min}}{\epsilon}\right),$ where $L_{\min}$ is the minimum smoothness parameter of the two blocks. 
This means that one of the two blocks can have arbitrarily poor smoothness (including $\infty$), as long it is easy to optimize over it. Some important examples with a nonsmooth block (with smoothness parameter equal to infinity) can be found in~\cite{beck2015convergence-AM}. Additional examples of problems for which exact optimization over the least smooth block can be performed efficiently are provided in Appendix~\ref{app:efficient-iterations}.

In this paper, we address the following open question, which was implicitly raised by~\cite{beck2013convergence-BCD}: can we design algorithms that combine the features of randomized block-coordinate descent and alternating minimization? In particular, assuming we can perform exact optimization on block $n$, can we construct a block-coordinate descent algorithm whose running time scales with $O(\sum_{i=1}^{n-1} L_i),$ i.e., independently of the smoothness $L_n$ of the $n^{\mathrm{th}}$ block?  This would generalize both existing block-coordinate descent methods, by allowing one block to be optimized exactly, and existing alternating minimization methods, by allowing $n$ to be larger than $2$ and requiring exact optimization only on a single block.

We answer these questions in the affirmative by presenting a novel algorithm: alternating randomized block coordinate descent (\ref{eq:AR-BCD}). The algorithm alternates between an exact optimization over a fixed, possibly non-smooth block, and a gradient descent or exact optimization over a randomly selected block among the remaining blocks. 
For two blocks, the method reduces to the standard alternating minimization, while when the non-smooth block is empty (not optimized over), we get randomized block coordinate descent (RCDM) from~\cite{nesterov2012efficiency}.

Our second contribution is \ref{eq:AAR-BCD}, an accelerated version of \ref{eq:AR-BCD}, which achieves the accelerated rate of $\frac{1}{k^2}$ without incurring any dependence on the smoothness of block $n$.  Furthermore,  when the non-smooth block is empty, \ref{eq:AAR-BCD} recovers the fastest known convergence bounds for block-coordinate descent~\cite{qu2016coordinate,allen2016even, nesterov2012efficiency, lin2014accelerated,nesterov-Stich2017efficiency}. As a special case,~\ref{eq:AAR-BCD} provides the first accelerated alternating minimization algorithm, obtained directly from~\ref{eq:AAR-BCD} when the number of blocks equals two.\footnote{The remarks about \emph{accelerated} alternating minimization have been added in the second version of the paper, in July 2019, partly to clarify the relationship to methods obtained in~\cite{guminov2019accelerated}, which was posted to the arXiv for the first time in June 2019. At a technical level, nothing new is introduced compared to the first version of the paper -- everything stated in Section~\ref{sec:acc-alt-min} follows either as a special case or a simple corollary of the results that appeared in the first version of the paper in May 2018.} %Details of efficient implementation of \ref{eq:AR-BCD} and examples of settings in which exact minimization over the least smooth block can be implemented efficiently are provided in Appendix~\ref{app:efficient-iterations}. 
Another conceptual contribution is our extension of the approximate duality gap technique of~\cite{thegaptechnique}, which leads to a general and more streamlined analysis. 

Finally, to illustrate the results, we perform a preliminary experimental evaluation of our methods against existing block-coordinate algorithms and discuss how their performance depends on the smoothness and size of the blocks.

\paragraph{Related Work} 
Alternating minimization and cyclic block coordinate descent are old and fundamental algorithms~\cite{ortega1970iterative} whose convergence (to a stationary point) has been studied even in the non-convex setting, in which they were shown to converge asymptotically under the additional assumptions that the blocks are optimized exactly and their minimizers are unique~\cite{bertsekas1999nonlinear}. However, even in the non-smooth convex case, methods that perform exact minimization over a fixed set of blocks may converge arbitrarily slowly. This has lead scholars to focus on the case of smooth convex minimization, for which nonasymptotic convergence rates were obtained recently in~\cite{beck2013convergence-BCD,beck2015convergence-AM,sun2015improved,saha2013nonasymptotic}. However, prior to our work, convergence bounds that are independent of the largest smoothness parameter were only known for the setting of two blocks.

%HERE PARAGRAPH ON BCD algorithms
Randomized coordinate descent methods, in which steps over coordinate blocks are taken in a non-cyclic randomized order (i.e., in each iteration one block is sampled with replacement) were originally analyzed in~\cite{nesterov2012efficiency}. The same paper~\cite{nesterov2012efficiency} also provided an accelerated version of these methods. The results of~\cite{nesterov2012efficiency} were subsequently improved and generalized to various other settings (such as, e.g., composite minimization) in \cite{lee2013efficient,allen2016even,nesterov-Stich2017efficiency,richtarik2014iteration,fercoq2015accelerated,lin2014accelerated}. The analysis of the different block coordinate descent methods under various sampling probabilities (that, unlike in our setting, are non-zero over all the blocks) was unified in \cite{qu2016coordinate} and extended to a more general class of steps within each block in \cite{ GowerR15, qu2016sdna}. 
 
Our results should be carefully compared to a number of proximal block-coordinate methods that rely on different assumptions~\cite{tseng2009coordinate,richtarik2014iteration,lin2014accelerated,fercoq2015accelerated}. In this setting, the function $f$ is assumed to have the structure $f_0(\vx)+ \Psi(\vx),$ where $f_0$ is smooth, the non-smooth convex function $\Psi$ is separable over the blocks, i.e., $\Psi(\vx) = \sum_{i=1}^n \Psi_i(\vx_i)$, and we can efficiently compute the proximal operator of each $\Psi_i$. This strong assumption allows these methods to make use of the standard proximal optimization framework.
By contrast, in our paper, the convex objective can be taken to have an arbitrary form, where the non-smoothness of a block need not be separable, though the function is assumed to be differentiable. 
%We expect that our ideas can be extended to the the case of composite objectives from~\cite{tseng2009coordinate,richtarik2014iteration,lin2014accelerated,fercoq2015accelerated}, where we would be able to additionally allow for the smoothness parameter of the $n^{\mathrm{th}}$ block of $f$ to be infinite.

\section{Preliminaries}\label{sec:prelims}
We assume that we are given oracle access to the gradients of a continuously differentiable convex function $f: \mathbb{R}^N \rightarrow \mathbb{R}$, where computing gradients over only a subset of coordinates is computationally much cheaper than computing the full gradient. We are interested in minimizing $f(\cdot)$ over $\mathbb{R}^N$, and we denote $\vx_* = \argmin_{\vx \in \mathbb{R}^N}f(\vx)$. We let $\|\cdot\|$ denote an arbitrary (but fixed) norm, and $\|\cdot\|_*$ denote its dual norm, defined in the standard way: $\|\vz\|_* = \sup_{\vx \in \mathbb{R}^N: \|\vx\|=1}\innp{\vz, \vx}$.\footnote{Note that the analysis extends in a straightforward way to the case where each block is associated with a different norm (see, e.g.,~\cite{nesterov2012efficiency}); for simplicity of presentation, we take the same norm over all blocks.} 

 Let $I_N$ be the identity matrix of size $N$, $I_N^{i}$ be a diagonal matrix whose diagonal elements $j$ are equal to one if variable $j$ is in the $i^{\mathrm{th}}$ block, and zero otherwise. Notice that $I_N = \littlesum_{i=1}^n I_N^i$.  Let $S_{i}(\vx) = \{\vy \in \mathbb{R}^N : (I_N - I_N^i)\vy = (I_N - I_N^i)\vx\}$, that is, $S_i$ contains all the points from $\mathbb{R}^N$ whose coordinates differ from those of $\vx$ only over block $i$.

We denote the smoothness parameter of block $i$ by $L_i$, as defined in Equation~\eqref{eq:smoothness-grad}. Equivalently,  
%\begin{equation}\label{eq:smoothness-grad}%\|\nabla_i f(\vx) - \nabla_i f(\vy)\|_* \leq L_i \|\vx^i - \vy^i\|, \forall \vx, \vy \in \mathbb{R}^N,
%\end{equation}
%or, equivalently, 
$\forall \vx, \vy \in \mathbb{R}^N$:
\begin{equation}\label{eq:smoothness-2nd-order}
f(\vx + I_N^i \vy) \leq f(\vx) + \innp{\nabla_i f(\vx), \vy^i} + \frac{L_i}{2}\|\vy^i\|^2.
\end{equation}

The gradient step over block $i$ is then defined as:
\begin{equation}\label{eq:grad-step-i}
\begin{aligned}
&T_i(\vx)\\
&\hspace{.2cm}= \argmin_{\vy \in {S}_i(\vx)} \Big\{ \innp{\nabla f(\vx), \vy - \vx} + \frac{L_i}{2}\|\vy - \vx\|^2 \Big\}.
\end{aligned}
\end{equation}
By standard arguments (see, e.g., Exercise 3.27 in \cite{boyd2004convex}):
\begin{equation}\label{eq:gradient-progress}
f(T_i(\vx))-f(\vx) \leq -\frac{1}{2L_i}\|\nabla_i f(\vx)\|_*^2.
\end{equation}

Without loss of generality, we will assume that the $n^{\mathrm{th}}$ block has the largest smoothness parameter and is possibly non-smooth (i.e., it can be $L_n = \infty$). The standing assumption is that exact minimization over the $n^{\mathrm{th}}$ block is ``easy'', meaning that it is computationally inexpensive and possibly solvable in closed form; for some important examples that have this property, see Appendix~\ref{app:efficient-iterations}. Observe that when block $n$ contains a small number of variables, it is often computationally inexpensive to use second-order optimization methods, such as, e.g., interior point method. %Due to space constraints, we present the results assuming that the function is block-smooth, but not strongly convex. The results for block-smooth and strongly convex functions can be found in the supplementary material. 

We assume that $f(\cdot)$ is strongly convex with parameter $\mu \geq 0$, where it could be $\mu = 0$ (in which case $f(\cdot)$ is not strongly convex). Namely, $\forall \vx, \vy$:
\begin{equation}\label{eq:strong-convexity}
f(\vy) \geq f(\vx) + \innp{\nabla f(\vx), \vy - \vx} + \frac{\mu}{2}\|\vy - \vx\|^2.
\end{equation}
When $\mu > 0$, we take $\|\cdot\| = \|\cdot\|_2$, which is customary for smooth and strongly convex minimization~\cite{Bube2014}. 

Throughout the paper, whenever we take unconditional expectation, it is with respect to all randomness in the algorithm.% (which includes all random choices of coordinates $i_k$ over iterations $k$).

\subsection{Alternating Minimization}
In (standard) alternating minimization (AM), there are only two blocks of coordinates, i.e., $n = 2$. The algorithm is defined as follows.
\noindent\vsepfbox{\begin{minipage}{.47\textwidth}
\begin{equation}\label{eq:alt-min}\tag{AM}
\begin{gathered}
\vxh_{k} = \argmin_{\vx \in S_1(\vx_{k-1})} f(\vx),\\
\vx_{k} = \argmin_{\vx \in S_2(\vxh_k)} f(\vx),\\
\vx_1 \in \mathbb{R}^N \text{ is an arbitrary initial point.}
\end{gathered}
\end{equation}
\end{minipage}}
We note that for the standard analysis of alternating minimization~\cite{beck2015convergence-AM}, the exact minimization step over the smoother block can be replaced by a gradient step (Equation~(\ref{eq:grad-step-i})), while still leading to convergence that is only dependent on the smaller smoothness parameter. 

\subsection{Randomized Block Coordinate (Gradient) Descent}

The simplest version of randomized block coordinate (gradient) descent (RCDM) can be stated as~\cite{nesterov2012efficiency}:

\noindent\vsepfbox{\begin{minipage}{.47\textwidth}
\begin{equation}\label{eq:standard-BCD}\tag{RCDM}
\begin{gathered}
\text{Select } i_k\in \{1,\dots,n\} \text{ w.p. }p_{i_k} > 0,\\
\vx_{k} = T_{i_k}(\vx_{k-1}),\\
\vx_1\in \mathbb{R}^N \text{ is an arbitrary initial point,}
\end{gathered}
\end{equation}
\end{minipage}}
where $\sum_{i=1}^n p_i = 1$. 
A standard choice of the probability distribution is $p_i \sim L_i$, leading to the convergence rate that depends on the sum of block smoothness parameters.

\section{AR-BCD}\label{sec:standard-algo}

The basic version of alternating randomized block coordinate descent (AR-BCD) is a direct generalization of (\ref{eq:alt-min}) and (\ref{eq:standard-BCD}): when $n=2$, it is equivalent to (\ref{eq:alt-min}), while when the size of the $n^{\mathrm{th}}$ block is zero, it reduces to (\ref{eq:standard-BCD}). The method is stated as follows:
\noindent\vsepfbox{
\begin{minipage}{.47\textwidth}
\begin{equation}\label{eq:AR-BCD}\tag{AR-BCD}
\begin{gathered}
\text{Select } i_k\in \{1,\dots,n-1\} \text{ w.p. }p_{i_k} > 0,\\
\vxh_k = T_{i_k}(\vx_{k-1}),\\
\vx_{k} = \argmin_{\vx \in S_n(\vxh_k)}f(\vx),\\
\vx_1\in \mathbb{R}^N \text{ is an arbitrary initial point,}
\end{gathered}
\end{equation}
\end{minipage}}
where $\sum_{i=1}^{n-1} p_i = 1$. We note that nothing will change in the analysis if the step $\vxh_{k} = T_{i_k}(\vx_{k-1})$ is replaced by $\vxh_{k} = \argmin_{\vx \in S_{i_k}(\vx_{k-1})}f(\vx)$, since $\min_{\vx \in S_{i_k}(\vx_{k-1})}f(\vx) \leq f(T_{i_k}(\vx_{k-1}))$. 

In the rest of the section, we show that (\ref{eq:AR-BCD}) leads to a convergence bound that interpolates between the convergence bounds of (\ref{eq:alt-min}) and (\ref{eq:standard-BCD}): it depends on the sum of the smoothness parameters of the first $n-1$ blocks, while the dependence on the remaining problem parameters is the same for all these methods.

\subsection{Approximate Duality Gap}

To analyze (\ref{eq:AR-BCD}), we extend the approximate duality gap technique~\cite{thegaptechnique} to the setting of randomized block coordinate descent methods. The approximate duality gap $G_k$ is defined as the difference of an upper bound $U_k$ and a lower bound $L_k$ to the minimum function value $f(\vx_*)$. For (\ref{eq:AR-BCD}), we choose the upper bound to simply be $U_k = f(\vx_{k+1})$. 

The generic construction of the lower bound is as follows. Let $\vx_1, \vx_2,..., \vx_k$ be any sequence of points from $\mathbb{R}^N$ (in fact we will choose them to be exactly the sequence constructed by (\ref{eq:AR-BCD})). Then, by (strong) convexity of $f(\cdot)$, $f(\vx_*)\geq f(\vx_j) + \innp{\nabla f(\vx_j), \vx_* - \vx_j} + \frac{\mu}{2}\|\vx_* - \vx_j\|^2$, $\forall j \in \{1, \dots, k\}$. In particular, if $a_j > 0$ is a sequence of (deterministic, independent of $i_j$) positive real numbers and $A_k = \sum_{j=1}^k a_j$, then:
\begin{align}
f(\vx_*) \geq & \frac{\littlesum_{j=1}^k a_j f(\vx_j) + \littlesum_{j=1}^k a_j \innp{\nabla f(\vx_j), \vx_* - \vx_j}}{A_k}\notag\\
& + \frac{\frac{\mu}{2}\sum_{j=1}^k a_j\|\vx_* - \vx_j\|^2}{A_k}\defeq  L_k. \label{eq:lb-non-acc}
\end{align}

\subsection{Convergence Analysis} 

The main idea in the analysis is to show that $\mathbb{E}[A_k G_k - A_{k-1}G_{k-1}] \leq E_k$, for some deterministic $E_k$. Then, using linearity of expectation, $\mathbb{E}[f(\vx_{k+1})] - f(\vx_*)\leq \mathbb{E}[G_k] \leq \frac{\mathbb{E}[A_1 G_1]}{A_k} + \frac{\sum_{j=2}^k E_j}{A_k}$. %$A_k G_k = A_k (U_k - L_k)$ is a supermartingale. Then, as $A_k$ is a deterministic sequence, $U_k = f(\vx_{k+1})$ and $L_k \leq f(\vx_*)$, we would immediately get that $\mathbb{E}[f(\vx_{k+1})]-f(\vx_*)\leq \mathbb{E}[G_k] \leq \frac{\mathbb{E}[A_1G_1]}{A_k}$. 
The bound in expectation can then be turned into a bound in probability, using well-known concentration bounds. The main observation that allows us not to pay for the non-smooth block is:
\begin{observation}\label{obs:grad-ns-zero}
For $\vx_k$'s constructed by (\ref{eq:AR-BCD}), $\nabla_n f(\vx_k) = \zeros$, $\forall k$, where $\zeros$ is the vector of all zeros.
\end{observation}
This observation is essentially what allows us to sample $i_k$ only from the first $n-1$ blocks, and holds due to the step $\vx_{k} = \argmin_{\vx \in S_n(\vxh_k)}f(\vx)$ from (\ref{eq:AR-BCD}).

Denote $R_{\vx_*^i} = \max_{\vx\in \mathbb{R}^N}\{\|I_N^i(\vx_* - \vx)\|^2: f(\vx) \leq f(\vx_1)\}$, and let us bound the initial gap $A_1G_1$. %We note that in the case when $\|\cdot\| = \|\cdot\|_2,$ we can use $R_{\vx_*^i} = \|I_N^i(\vx_* - \vx_1)\|_2^2,$ as $\|\vx_* - \vx_{k+1}\| \leq \|\vx_* - \vx_k\|$ (see, e.g., proof of Theorem 2.1.14 in~\cite{nesterov2013introductory}).
\begin{proposition}\label{prop:ARBCD-initial-gap}
$\mathbb{E}[A_1G_1] \leq E_1$, where $E_1 = a_1\littlesum_{i=1}^{n-1} \left(\frac{L_i}{2p_i}-\frac{\mu}{2}\right)R_{\vx_*^i}$.
\end{proposition}
\begin{proof}
By linearity of expectation, $\mathbb{E}[A_1G_1] = \mathbb{E}[A_1U_1] - \mathbb{E}[A_1L_1].$ The initial lower bound is deterministic, and, by  $\nabla_n f(\vx_1) = \zeros$ and duality of norms, is bounded as:
\begin{equation*}
\begin{aligned}
\mathbb{E}[A_1L_1]\geq & a_1f(\vx_1) - a_1 \littlesum_{i=0}^{n-1} \|\nabla_i f(\vx_1)\|_*\|\vx_*^i - \vx_1^i\|\\
& + a_1\frac{\mu}{2}\|\vx_* - \vx_1\|^2.
\end{aligned}
\end{equation*}
Using (\ref{eq:gradient-progress}), if $i_2 = i$, then:
\begin{equation*}
U_1 = f(\vx_2) \leq f(\vxh_2) \leq f(\vx_1) - \frac{1}{2L_i}\|\nabla_i f(\vx_1)\|_*^2.
\end{equation*}
Since block $i$ is selected with probability $p_i$ and $A_1 = a_1$:
\begin{align*}
\mathbb{E}[A_1 U_1] \leq & a_1 f(\vx_1) - \littlesum_{i=1}^{n-1} \frac{a_1p_i}{2L_i}\|\nabla_i f(\vx_i)\|_*^2. 
\end{align*}
Since the inequality $2ab - a^2 \leq b^2$ holds $\forall a, b$, we have:
\begin{equation*}
\begin{aligned}
& a_1 \|\nabla_i f(\vx_1)\|_*\|\vx_*^i - \vx_1^i\| - \frac{a_1p_i}{2L_i}\|\nabla_i f(\vx_i)\|_*^2\\
&\hspace{1cm}\leq \frac{a_1L_i}{2p_i}\|\vx_*^i - \vx_1^i\|^2, \; \forall i \in \{1,\dots, n-1\}
\end{aligned}
\end{equation*}
Hence, when $\mu = 0$, $\mathbb{E}[A_1 G_1] \leq \sum_{i=1}^{n-1}\frac{a_1L_i}{2p_i}\|\vx_*^i - \vx_1^i\|^2$. When $\mu > 0,$ since in that case we are assuming $\|\cdot\| = \|\cdot\|_2$ (Section~\ref{sec:prelims}), $\|\vx_* - \vx_1\|^2 \geq \sum_{i=1}^{n-1}\|\vx_*^i - \vx_1^i\|^2$, leading to $\mathbb{E}[A_1 G_1] \leq a_1\sum_{i=1}^{n-1}\left(\frac{L_i}{2p_i}- \frac{\mu}{2}\right)\|\vx_*^i - \vx_1^i\|^2$.
%Combining the bounds on $\mathbb{E}[A_1U_1]$ and $\mathbb{E}[A_1L_1]$, the proof follows.
\end{proof}

We now show how to bound the error in the decrease of the scaled gap $A_kG_k$.

\begin{lemma}\label{lemma:ARBCD-gap-decrease}
$\mathbb{E}[A_kG_k - A_{k-1}G_{k-1}] \leq E_k$, where $E_k = a_k\sum_{i=1}^{n-1}\left(\frac{{a_k}L_i}{2A_kp_i}-\frac{\mu}{2}\right)R_{\vx_*^i}$.
\end{lemma}
\begin{proof}
Let $\mathcal{F}_k$ denote the natural filtration up to iteration $k$. By linearity of expectation and $A_{k}L_k - A_{k-1}L_{k-1}$ being measurable w.r.t. $\mathcal{F}_k$, 
\begin{align*}
&\mathbb{E}[A_kG_k - A_{k-1}G_{k-1}|\mathcal{F}_{k}]\\
&\hspace{.2cm}= \mathbb{E}[A_kU_k - A_{k-1}U_{k-1}|\mathcal{F}_{k}] - (A_k L_k - A_{k-1}L_{k-1}).
\end{align*}
With probability $p_i$ and as $f(\vx_{k+1}) \leq f(\vxh_{k+1})$, the change in the upper bound is:
\begin{align}
A_k U_k - A_{k-1}U_{k-1} \leq & A_k f(\vxh_{k+1}) - A_{k-1}f(\vx_k)\notag\\
\leq &  a_k f(\vx_k) - \frac{A_k}{2L_i}\|\nabla_i f(\vx_k)\|_*^2\notag,
\end{align}
where the second line follows from $\vxh_{k+1} = T_{i_k}(\vx_k)$ and Equation (\ref{eq:gradient-progress}). Hence:
\begin{align*}
&\mathbb{E}[A_k U_k - A_{k-1}U_{k-1}|\mathcal{F}_{k}]\\
&\hspace{1cm}\leq a_k f(\vx_k) - A_k\littlesum_{i=1}^{n-1}\frac{p_i}{2L_i}\|\nabla_i f(\vx_k)\|_*^2.
\end{align*}
On the other hand, using the duality of norms, the change in the lower bound is:
\begin{align*}
&A_kL_k - A_{k-1}L_{k-1}\\
&\hspace{1cm}\geq a_k f(\vx_k) - a_k \littlesum_{i=1}^{n-1} \|\nabla_i f(\vx_k)\|_*\|\vx_*^i - \vx_k^i\|\\
&\hspace{1.4cm}+ a_k\frac{\mu}{2}\|\vx_* - \vx_k\|^2\\
&\hspace{1cm}\geq a_k f(\vx_k) - a_k \littlesum_{i=1}^{n-1} \|\nabla_i f(\vx_k)\|_* \sqrt{R_{\vx_*^i}}\\
&\hspace{1.4cm}+ a_k\frac{\mu}{2}\|\vx_* - \vx_k\|^2.
\end{align*}
By the same argument as in the proof of Proposition~\ref{prop:ARBCD-initial-gap}, it follows that: $\mathbb{E}[A_k G_k - A_{k-1}G_{k-1}|\mathcal{F}_k] \leq a_k\sum_{i=1}^{n-1}\left(\frac{L_i{a_k}}{2A_kp_i}-\frac{\mu}{2}\right)R_{\vx_*^i} = E_k$. Taking expectations on both sides, as $E_k$ is deterministic, the proof follows.
\end{proof}

We are now ready to prove the convergence bound for (\ref{eq:AR-BCD}), as follows.

\begin{theorem}\label{thm:ARBCD-convergence}
Let $\vx_k$ evolve according to (\ref{eq:AR-BCD}). Then, $\forall k \geq 1$:
\begin{enumerate}
\item If $\mu = 0:$
$
\mathbb{E}[f(\vx_{k+1})] - f(\vx_*) \leq \frac{2\littlesum_{i=1}^{n-1}\frac{L_i}{p_i}R_{\vx_*^i}}{k+3}.
$ 
In particular, for $p_i = \frac{L_i}{\sum_{i' = 1}^{n-1} L_{i'}},$ $1\leq i \leq n-1$:
$$
\mathbb{E}[f(\vx_{k+1})] - f(\vx_*) \leq \frac{2(\sum_{i'=1}^{n-1}L_{i'}) \littlesum_{i=1}^{n-1}R_{\vx_*^i}}{k+3}.
$$
Similarly, for $p_i = \frac{1}{n-1}$, $1\leq i\leq n-1:$
$$
\mathbb{E}[f(\vx_{k+1})] - f(\vx_*) \leq \frac{2(n-1)\littlesum_{i=1}^{n-1}{L_i}R_{\vx_*^i}}{k+3}
$$
\item If $\mu > 0$, $p_i = \frac{L_i}{\sum_{i' = 1}^{n-1}L_{i'}}$ and $\|\cdot\| = \|\cdot\|_2:$
\begin{align*}
&\mathbb{E}[f(\vx_{k+1})] - f(\vx_*)\\
&\hspace{1cm}\leq \Big(1 - \frac{\mu}{\sum_{i'=1}^{n-1}L_{i'}}\Big)^k\\
& \hspace{1.5cm}\cdot \frac{(\sum_{i'=1}^{n-1}L_{i'})\|(I_N - I_N^n)(\vx_* - \vx_1)\|^2}{2}.
\end{align*}
\end{enumerate}
\end{theorem}
\begin{proof}
From Proposition~\ref{prop:ARBCD-initial-gap} and Lemma~\ref{lemma:ARBCD-gap-decrease}, by linearity of expectation and the definition of $G_k$:
\begin{equation}
\mathbb{E}[f(\vx_{k+1})] - f(\vx_*) \leq \mathbb{E}[G_k] \leq \frac{\sum_{j=1}^k E_j}{A_k},
\end{equation}
where $E_j = \frac{{a_j}^2}{A_j}\sum_{i=1}^{n-1}\frac{L_i}{2p_i}R_{\vx_*^i}$.

Notice that the algorithm does not depend on the sequence $\{a_j\}$ and thus we can choose it arbitrarily. Suppose that $\mu = 0$. Let $a_j = \frac{j+1}{2}$. Then $\frac{{a_j}^2}{A_j} = \frac{(j+1)^2}{j(j+3)}\leq 1$, and thus:
$%\begin{equation*}
\frac{\sum_{j=1}^k E_j}{A_k} \leq \frac{2\littlesum_{i=1}^{n-1}\frac{L_i}{p_i}R_{\vx_*^i}}{k+3}, 
$%\end{equation*}
which proves the first part of the theorem, up to concrete choices of $p_i$'s, which follow by simple computations.

For the second part of the theorem, as $\mu > 0$, we are assuming that $\|\cdot\| = \|\cdot\|_2$, as discussed in Section~\ref{sec:prelims}. From Lemma~\ref{lemma:ARBCD-gap-decrease}, $E_j = a_j \sum_{i=1}^{n-1} \left(\frac{a_j L_i}{2A_jp_i} - \frac{\mu}{2}\right)R_{\vx_*^i}$, $\forall j \geq 2$. As $p_i = \frac{L_i}{\sum_{i'=1}^{n-1}L_{i'}}$, if we take $\frac{a_j}{A_j} = \frac{\mu}{\sum_{i'=1}^{n-1}L_{i'}}$, it follows that $E_j = 0$, $\forall j \geq 2$. Let $a_1 = A_1 = 1$ and $\frac{a_j}{A_j} = \frac{\mu}{\sum_{i'=1}^{n-1}L_{i'}}$ for $j \geq 2$. Then: 
$%\begin{align*}
\mathbb{E}[f(\vx_{k+1})] - f(\vx_*) \leq \mathbb{E}[G_k]
\leq  \frac{\mathbb{E}[A_1 G_1]}{A_k}.
$ %\end{align*}
As $\frac{A_1}{A_k} = \frac{A_1}{A_2}\cdot \frac{A_2}{A_3}\cdot \dots \cdot \frac{A_{k-1}}{A_k}$ and $\frac{A_{j-1}}{A_j} = 1 - \frac{a_j}{A_j}$: 
$%\begin{align*}
\mathbb{E}[f(\vx_{k+1})] - f(\vx_*) 
\leq  \Big(1 - \frac{\mu}{\sum_{i'=1}^{n-1}L_{i'}}\Big)^{k-1}\mathbb{E}[G_1].
$ %\end{align*}
It remains to observe that, from Proposition~\ref{prop:ARBCD-initial-gap}, $\mathbb{E}[G_1] \leq \big(1 - \frac{\mu}{\sum_{i'=1}^{n-1}L_{i'}}\big)\frac{(\sum_{i'=1}^{n-1}L_{i'})\|(I_N - I_N^n)(\vx_* - \vx_1)\|^2}{2}$.
\end{proof}

%When $\|\cdot\| = \|\cdot\|_2$, as already discussed, we can replace $\sum_{i=1}^{n-1}R_{\vx_*^i}$ by $\|(I_N - I_N^n)(\vx_* - \vx_1)\|^2$.
%Hence, for $\|\cdot\| = \|\cdot\|_2$ and $p_i = \frac{L_i}{\sum_{i'=1}^{n-1}L_{i'}}$, $i \in \{1,\dots, n\}$, (\ref{eq:AR-BCD}) converges to error $\frac{2\left(\sum_{i=1}^{n-1} L_i\right)\|(I_N - I_N^n)(\vx_* - \vx_1)\|^2}{k+1}$ in $k+1$ steps. 
We note that when $n=2$, the asymptotic convergence of~\ref{eq:AR-BCD}  coincides with the convergence of alternating minimization~\cite{beck2015convergence-AM}. When $n^{\mathrm{th}}$ block is empty (i.e., when all blocks are sampled with non-zero probability and there is no exact minimization over a least-smooth block), we obtain the convergence bound of the standard randomized coordinate descent method~\cite{nesterov2012efficiency}.
%%%%%%%%%%%%%%%%%%%%%%%%%%%% ACCELERATED AR-BCD
%%%%%%%%%%%%%%%%%%%%%%%%%%%%%%%%%%%%%%%%%%%%%%%
\section{Accelerated AR-BCD}\label{sec:acc-algo}

In this section, we show how to accelerate (\ref{eq:AR-BCD})  when $f(\cdot)$ is smooth. We believe it is possible to obtain similar results in the smooth and strongly convex case, which we defer to a future version of the paper. %We also take $\|\cdot\| = \|\cdot\|_2$. This choice of the norm is important for efficient implementation of an iteration described in Appendix~\ref{app:efficient-iterations}.
%
%We start by introducing some additional notation. 
%Throughout the section, we assume that we are given a function $\psi(\cdot)$ that is separable over blocks $i \in \{1,\dots, n-1\}$, that is $\psi(\vx) = \sum_{i=1}^{n-1}\psi_i(\vx^i)$. Further, we assume that each $\psi_i(\cdot)$ is $\sigma_i$-strongly convex with respect to the block norm $\|\cdot\|$, that is:
%\begin{equation*}
%\psi_i(\vy^i)\geq \psi_i(\vx^i) + \innp{\nabla \psi_i(\vx^i), \vy^i - \vx^i} + \frac{\sigma_i}{2}\|\vy^i - \vx^i\|^2
%\end{equation*}
%and that problems of the form $\min_{\vx^i}\{\innp{\vz^i, \vx^i} + \psi_i(\vx^i)\}$ are easily solvable. %In the case when $\mu > 0$ (and $\|\cdot\| = \|\cdot\|_2$), we take $\psi_i(\vx) = \frac{\sigma_i}{2}\|\vx\|^2$. 
Denote:
\begin{gather}
\Delta_k = I_N^{i_k} \nabla f(\vx_k)/p_{i_k},\label{eq:Delta_k-def}\notag\\
\vv_k = \argmin_{\vu}\Big\{\littlesum_{j=1}^k a_j \innp{\Delta_j, \vu}% + \frac{\mu}{2}\|\vu - \vx_j\|^2) \notag 
\notag\\
\hspace{3cm}
+ \littlesum_{i=1}^n \frac{\sigma_i}{2}\|\vu^i - \vx_1^i\|^2\Big\},\label{eq:arg-min-lb}
\end{gather}
where $\sigma_i > 0$, $\forall i$, will be specified later.
%$D_{\psi}(\vu, \vx_1)$ denotes Bregman divergence, defined in the standard way: $D_{\psi}(\vy, \vx) = \psi(\vy) - \psi(\vx) - \innp{\nabla \psi(\vx), \vy - \vx}$, $\forall \vy, \vx \in \mathbb{R}^N$.  
Accelerated AR-BCD (AAR-BCD) is defined as follows:
\vsepfbox{
\begin{minipage}{.47\textwidth}
\begin{equation}\label{eq:AAR-BCD}\tag{AAR-BCD}
\begin{gathered}
\text{Select }i_k \text{ from } \{1,\dots, n-1\} \text{ w.p. } p_{i_k},\\
\vxh_k = \frac{A_{k-1}}{A_k}\vy_{k-1} + \frac{a_k}{A_k}\vv_{k-1},\\
\vx_k = \argmin_{\vx \in S_n(\vxh_k)}f(\vx),\\
\vy_k = \vx_k + \frac{a_k}{p_{i_k}A_k}I_N^{i_k}(\vv_{k}-\vv_{k-1}),\\%T_{i_k}(\vx_k),\\%
\vx_1 \text{ is an arbitrary initial point,}
\end{gathered}
\end{equation}
\end{minipage}}
where $\sum_{i=1}^{n-1}p_i = 1$, $p_i > 0$, $\forall i \in \{1,\dots, n-1\}$, and $\vv_k$ is defined by (\ref{eq:arg-min-lb}). To seed the algorithm, we further assume that %when $\mu > 0$, $\vy_1 = \vx_1 + I_N^{i_1}\frac{1}{\sqrt{p_{i_1}}}(\vv_1 - \vx_1)$, while when $\mu = 0$, 
$\vy_1 = \vx_1 + I_N^{i_1}\frac{1}{{p_{i_1}}}(\vv_1 - \vx_1)$. 

\begin{remark}\label{remark:iteration-complexity}
Iteration complexity of (\ref{eq:AAR-BCD}) is dominated by the computation of $\vxh_k,$ which requires updating an entire vector. This type of an update is not unusual for accelerated block coordinate descent methods, and in fact appears in all such methods we are aware of~\cite{nesterov2012efficiency,lee2013efficient,lin2014accelerated,fercoq2015accelerated,allen2016even}. In most cases of practical interest, however, it is possible to implement this step efficiently (using that $\vv_k$ changes only over block $i_k$ in iteration $k$). More details are provided in Appendix~\ref{app:efficient-iterations}.%For more information, see, e.g.,~\cite{lee2013efficient,fercoq2015accelerated,lin2014accelerated}.
\end{remark}

To analyze the convergence of \ref{eq:AAR-BCD}, we will need to construct a more sophisticated duality gap than in the previous section, as follows.

\subsection{Approximate Duality Gap}

\begin{figure*}[ht!]%\label{eq:Lambda-k}
\begin{equation}\label{eq:rand-lb}
\Lambda_k = \frac{\sum_{j=1}^k a_j f(\vx_j) + \min_{\vu \in \mathbb{R}^N}\left\{\sum_{j=1}^k a_j\innp{\Delta_j, \vu - \vx_j}
+ \littlesum_{i=1}^{n-1} \frac{\sigma_i}{2}\|\vu^i - \vx_1^i\|^2\right\}- \littlesum_{i=1}^{n-1} \frac{\sigma_i}{2}\|\vx_*^i - \vx_1^i\|^2}{A_k}.
\end{equation}
\vspace{-15pt}
\end{figure*}

We define the upper bound to be $U_k = f(\vy_k)$. 
The constructed lower bound $L_k$ from previous subsection is not directly useful for the analysis of (\ref{eq:AAR-BCD}). Instead, we will construct a random variable $\Lambda_k$, which in expectation is upper bounded by $f(\vx^*)$. The general idea, as in previous subsection, is to show that some notion of approximate duality gap decreases in expectation. 
%Using standard concentration bounds, the bound in expectation can then be turned into a bound in probability  to show that after sufficiently many iterations w.h.p. the method converges to a point $\vx_k$ with a small approximation error $f(\vy_k)-f(\vx_*)$.

Towards constructing $\Lambda_k,$ we first prove the following technical proposition, whose proof is in Appendix~\ref{app:omitted-proofs}.
\begin{proposition}\label{prop:expect-Delta_k}
Let $\vx_k$ be as in (\ref{eq:AAR-BCD}). Then: $$
\mathbb{E}[\littlesum_{j=1}^k a_j \innp{\Delta_j, \vx_* - \vx_j}] = \mathbb{E}[\littlesum_{j=1}^k a_j \innp{\nabla f(\vx_j), \vx_* - \vx_j}].
$$
\end{proposition}
Define the randomized lower bound as in Eq.~(\ref{eq:rand-lb}), and observe that (\ref{eq:arg-min-lb}) defines $\vv_k$ as the argument of the minimum from $\Lambda_k$. 
The crucial property of $\Lambda_k$ is that it lower bounds $f(\vx_*)$ in expectation, as shown in the following lemma.

\begin{lemma}\label{lemma:Lambda_k-is-lb}
Let $\vx_k$ be as in (\ref{eq:AAR-BCD}). Then
$
f(\vx_*) \geq \mathbb{E}[\Lambda_k].
$
\end{lemma}
\begin{proof}%[Proof of Lemma~\ref{lemma:Lambda_k-is-lb}]
By convexity of $f(\cdot),$ for any sequence  $\{\tilde{\vx}_j\}$ from $\mathbb{R}^N$, $f(\vx_*)\geq \frac{\sum_{j=1}^k a_j (f(\tilde{\vx}_j) + \innp{\nabla f(\tilde{\vx}_j), \vx_* - \tilde{\vx}_j})}{A_k}$. Since the statement holds for any sequence $\{\tilde{\vx}_j\},$ it also holds if $\{\tilde{\vx}_j\}$ is selected according to some probability distribution. In particular, for $\{\tilde{\vx}_j\} = \{\vx_j\}$: 
\begin{align*}
f(\vx_*)\geq &\mathbb{E}\Big[\frac{\sum_{j=1}^k a_j (f(\vx_j) + \innp{\nabla f(\vx_j), \vx_* - \vx_j})}{A_k}\Big].
\end{align*}
By linearity of expectation and Proposition~\ref{prop:expect-Delta_k}:
\begin{equation}\label{eq:lb-in-exp}
\begin{aligned}
f(\vx_*)\geq \mathbb{E}\Big[\frac{\sum_{j=1}^k a_j (f(\vx_j) + \innp{\Delta_j, \vx_* - \vx_j})}{A_k} \Big].
\end{aligned}
\end{equation}
Adding and subtracting (deterministic) $\littlesum_{i=1}^{n-1} \frac{\sigma_i}{2}\|\vx_*^i - \vx_1^i\|^2$ to/from (\ref{eq:lb-in-exp}) and using that:
\begin{equation*}
\begin{aligned}
&\littlesum_{j=1}^k a_j \innp{\Delta_j, \vx_* - \vx_j} + \littlesum_{i=1}^{n-1} \frac{\sigma_i}{2}\|\vx_*^i - \vx_1^i\|^2\\
%&+ \littlesum_{j=1}^k a_j\frac{\mu}{2}\|\vx_* - {\vx_j}\|^2\\
&\hspace{.5cm}\geq \min_{\vu}\Big\{ \littlesum_{j=1}^k a_j \innp{\Delta_j, \vu - \vx_j}+\littlesum_{i=1}^{n-1} \frac{\sigma_i}{2}\|\vu^i - \vx_1^i\|^2%\\
%&\hspace{2cm}+ \littlesum_{j=1}^k a_j\frac{\mu}{2}\|\vx_* - {\vx_j}\|^2
\Big\}\\
&\hspace{.5cm}= \min_{\vu} m_k(\vu), 
\end{aligned}
\end{equation*}
where $m_k(\vu) = \littlesum_{j=1}^k a_j \innp{\Delta_j, \vu - \vx_j}+ \littlesum_{i=1}^{n-1} \frac{\sigma_i}{2}\|\vu^i - \vx_1^i\|^2%+ \littlesum_{j=1}^k a_j\frac{\mu}{2}\|\vx_*^i - {\vx_j^i}\|^2
$, it follows that:
\begin{align*}
&f(\vx_*) \geq \mathbb{E}\Big[\frac{\sum_{j=1}^k a_j f(\vx_j) - \littlesum_{i=1}^{n-1} \frac{\sigma_i}{2}\|\vx_*^i - \vx_1^i\|^2}{A_k}\\
&\hspace{1.7cm} + \frac{\min_{\vu \in \mathbb{R}^N} m_k(\vu)}{A_k}\Big],
\end{align*}
which is equal to $\mathbb{E}[\Lambda_k],$ and completes the proof.
\end{proof}
%In the rest of the section we fix $\{\vx_k\}$ to be the sequence of points generated by (\ref{eq:AAR-BCD}). 
%
Similar as before, define the approximate gap as $\Gamma_k = U_k - \Lambda_k$. 
Then, we can bound the initial gap as follows.
\begin{proposition}\label{prop:AAR-BCD-init-gap}
If $a_1 = \frac{{a_1}^2}{A_1} \leq \frac{\sigma_i{p_i}^2}{L_i}$, $\forall i\in\{1,..., n-1\}$, then $\mathbb{E}[A_1\Gamma_1] \leq \littlesum_{i=1}^{n-1} \frac{\sigma_i}{2}\|\vx_* - \vx_1\|^2$. %\textcolor{red}{To be updated}
\end{proposition}
\begin{proof}
As $a_1 = A_1$ and $\vy_1$ differs from $\vx_1$ only over block $i = i_1$, by smoothness of $f(\cdot)$:
\begin{align*}
&A_1 U_1 = A_1f(\vy_1)\\
&\leq a_1 f(\vx_1) + a_1\innp{\nabla_i f(\vx_1), \vy_1^i - \vx_1^i} + \frac{a_1 L_i}{2}\|\vy_1^i - \vx_1^i\|^2.
\end{align*}
On the other hand, the initial lower bound is:
\begin{align*}
A_1\Lambda_1 =& a_1 (f(\vx_1) + \innp{\Delta_1, \vv_1 - \vx_1}) \\
&+ \littlesum_{i=1}^{n-1} \frac{\sigma_i}{2}\|\vv_1^i - \vx_1^i\|^2
- \littlesum_{i=1}^{n-1} \frac{\sigma_i}{2}\|\vx_*^i - \vx_1^i\|^2.
\end{align*}
Recall that $\vy_1^i = \vx_1^i + \frac{1}{p_i}(\vv_1^i - \vx_1^i)$. Using  $A_1\Gamma_1  = A_1U_1 - A_1\Lambda_1$ and the bounds on $U_1,\, \Lambda_1$ from the above: 
$
A_1\Gamma_1 \leq  \littlesum_{i=1}^{n-1} \frac{\sigma_i}{2}\|\vx_*^i - \vx_1^i\|^2,
$
as $a_1 \leq {p_i}^2\frac{\sigma_i}{L_i}$, and, thus, $\mathbb{E}[A_1\Gamma_1] \leq  \littlesum_{i=1}^{n-1} \frac{\sigma_i}{2}\|\vx_*^i - \vx_1^i\|^2$. 
%
% If $\mu > 0$, then $\|\cdot\|= \|\cdot\|_2$ (by the assumptions from Section~\ref{sec:prelims}) and $\vy_1^i = \vx_1^i + \frac{1}{\sqrt{p_i}}(\vv_1^i - \vx_1^i)$. It follows that:
% \begin{align*}
% {A_1\Gamma_1} \leq &\littlesum_{i=1}^{n-1}\Big(\frac{{a_1}L_i}{2{p_{i}}}-\frac{\sigma_i + a_1\mu}{2}\Big)\|\vv_1^i - \vx_1^i\|^2\\
% &+ D_{\psi_*}(\vx_*, \vx_1)\\
% \leq & D_{\psi_*}(\vx_*, \vx_1),
% \end{align*}
% as $a_1 \leq \frac{(\sigma_i + a_1\mu){p_i}}{L_i}$, $\forall i$. Thus, $\mathbb{E}[A_1\Gamma_1] \leq  D_{\psi}(\vx_*, \vx_1)$.
\end{proof}

The next part of the proof is to show that $A_k\Gamma_k$ is a supermartingale. The proof is provided in Appendix~\ref{app:omitted-proofs}.

%\textcolor{red}{Work in progress below...}
\begin{lemma}\label{lemma:AkGammak-supermartingale}
If $\frac{{a_k}^2}{A_k}\leq \frac{{p_i}^2\sigma_i}{L_i}$, $\forall i \in\{1, \dots,n-1\}$, then $\mathbb{E}[A_k\Gamma_k|\mathcal{F}_{k-1}] \leq A_{k-1}\Gamma_{k-1}$.
\end{lemma} 

Finally, we bound the convergence of~(\ref{eq:AAR-BCD}).
\begin{theorem}\label{thm:AAR-BCD-convergence}
Let $\vx_k,\, \vy_k$ evolve according to (\ref{eq:AAR-BCD}), for $\frac{{a_k}^2}{A_k} = \min_{1\leq i\leq n-1}\frac{\sigma_i{p_i}^2}{L_i} = \mathrm{const}$. Then, $\forall k \geq 1$:
$$
\mathbb{E}[f(\vy_k)] - f(\vx_*) \leq \frac{\sum_{i=1}^{n-1} \sigma_i \|\vx_*^i - \vx_a^i\|^2}{2A_k}.
$$
In particular, if $p_i = \frac{\sqrt{L_i}}{\sum_{i'=1}^{n-1}\sqrt{L_{i'}}}$, $\sigma_i = (\sum_{i'=1}^{n-1}\sqrt{L_{i'}})^2$, and $a_1 = 1$, then:
$$
\mathbb{E}[f(\vy_k)] - f(\vx_*) \leq \frac{2(\sum_{i'=1}^{n-1}\sqrt{L_{i'}})^2 \sum_{i=1}^{n-1}\|\vx_*^i - \vx_1^i\|^2}{k(k+3)}.
$$
Alternatively, if $p_i = \frac{1}{n-1}$, $\sigma_i = L_i$, and $a_1 = \frac{1}{(n-1)^2}$:
$$
\mathbb{E}[f(\vy_k)] - f(\vx_*) \leq \frac{2(n-1)^2\sum_{i=1}^{n-1}L_i \|\vx_*^i - \vx_1^i\|^2}{k(k+3)}.
$$
\end{theorem}
\begin{proof}
The first part of the proof follows immediately by applying Proposition~\ref{prop:AAR-BCD-init-gap} and Lemma~\ref{lemma:AkGammak-supermartingale}. The second part follows by plugging in the particular choice of parameters and observing that $a_j$ grows faster than $\frac{j+1}{2}$ in the former, and faster than $\frac{j+1}{2(n-1)^2}$ in the latter case.%For the second and third part, it is simple to verify that the choice of parameters guarantees $\frac{{a_j}^2}{A_j} \leq \frac{{p_i}^2\sigma_i}{L_i}$, $\forall i \in \{1,\dots, n-1\}$, and the rest follows by computing $A_k$. 
%
%
%Now consider the case $\mu > 0$. It is not hard to verify that the chosen parameters satisfy the conditions from Proposition~\ref{prop:AAR-BCD-init-gap} and Lemma~\ref{lemma:AkGammak-supermartingale}. Solving the quadratic equation $\frac{{a_j}^2}{A_jA_{j-1}} = \sum_{i=1}^{n-1}L_i$ and showing that $1-\frac{a_j}{A_j} = \frac{A_{j-1}}{A_j}\leq 1-\frac{1}{2\sqrt{\kappa}}$, the rest of the proof follows.
\end{proof}

Finally, we make a few remarks regarding Theorem~\ref{thm:AAR-BCD-convergence}. In the setting without a non-smooth block (when $n^{\mathrm{th}}$ block is empty), (\ref{eq:AAR-BCD}) with sampling probabilities $p_i \sim \sqrt{L_i}$ has the same convergence bound as the NU\_ACDM algorithm~\cite{allen2016even} and the ALPHA algorithm for smooth minimization~\cite{qu2016coordinate}. Further, when the sampling probabilities are uniform, (\ref{eq:AAR-BCD}) converges at the same rate as the ACDM algorithm \cite{nesterov2012efficiency} and the APCG algorithm applied to non-composite functions \cite{lin2014accelerated}.

\subsection{Accelerated Alternating Minimization}\label{sec:acc-alt-min}

Before making specific remarks about accelerated alternating minimization (case $n = 2$), we first note that the convergence analysis of~\ref{eq:AAR-BCD} applies generically even if the step $\vy_k$ is replaced by exact minimization over block $i_k$ (namely, if, instead of the current of choice of $\vy_k$ in~\ref{eq:AAR-BCD} we set $\vy_k = \argmin_{\vx \in S_{i_k}(\vx_k)}f(\vx)$). This change is relevant only at the beginning of the proof of Lemma~\ref{lemma:AkGammak-supermartingale}, and, to see that the same analysis still applies, observe that:
$$
f(\vy_k) = \min_{\vx \in S_{i_k}(\vx_k)}f(\vx) \leq f(\vx_k + \frac{a_k}{p_{i_k}A_k}I_N^{i_k}(\vv_{k}-\vv_{k-1})).
$$
That is, replacing a particular step over block $i_k$ with the exact minimization over the same block can only reduce the function value, which can only make the upper bound $U_k$ (and, thus, the gap $\Gamma_k$) lower.

Accelerated alternating minimization with exact minimization over one block is immediately obtained from~\ref{eq:AAR-BCD} as a special case when $n = 2.$ To obtain a version of the method with exact minimization over both blocks, it simply suffices to replace the $\vy_k$ step with the exact minimization over the first block, and, as already discussed, the same analysis applies. 

To obtain a method that is symmetric over blocks 1 and 2, one only needs to replace the roles of blocks 1 and 2 in, say, even iterations. Again, the same analysis applies, except that in even iterations one would need to have $\frac{{a_k}^2}{A_k} \leq \frac{\sigma_2}{L_2},$ whereas in odd iterations it would still be $\frac{{a_k}^2}{A_k} \leq \frac{\sigma_1}{L_1};$\footnote{Observe that when $n = 2$, $\{p_{i}\}$ is supported on a single element -- block 1 is selected deterministically in even iterations; block 2 is selected deterministically in odd iterations.} this affects the convergence bound in Theorem~\ref{thm:AAR-BCD-convergence} by at most a factor of 2. Observe further that when $L_2 \rightarrow \infty,$ we have $a_k \rightarrow 0$ in even iterations, and the method reduces to the non-symmetric version obtained directly from~\ref{eq:AAR-BCD}. 

Finally, in the case of two blocks ($n = 2$), it is straightforward to obtain a parameter-free version of the method. Indeed, all that is needed for the proof is that $A_k \Gamma_k \leq A_{k-1}\Gamma_{k-1}$ (as in Lemma~\ref{lemma:AkGammak-supermartingale}. As discussed before, when $n = 2,$ there is no randomness in the algorithm. Suppose that we want to implement a parameter-free version of the symmetric method and consider odd iterations (which are obtained as special cases of~\ref{eq:AAR-BCD} for $n=2$, with or without the exact minimization in the $\vy_k$-step). Then, from Lemma~\ref{lemma:AkGammak-supermartingale}, all that needs to be satisfied is that $\frac{{a_k}^2}{A_k}\leq \frac{\sigma_1}{L_1}.$ Hence, if we find $a_k^*$ for which  $A_k \Gamma_k \leq A_{k-1}\Gamma_{k-1},$ then $A_k \Gamma_k \leq A_{k-1}\Gamma_{k-1}$ for any $a_k \leq a_k^*.$ This is sufficient for implementing a backtracking line search over $a_k$ to ensure $A_k \Gamma_k \leq A_{k-1}\Gamma_{k-1}.$ To do so, one can use:
\begin{align*}
    A_k \Gamma_k -& A_{k-1}\Gamma_{k-1}\\
    =& A_k f(\vy_k) - A_{k-1} f(\vy_{k-1})-a_k f(\vx_k)\\
    &-  a_k\innp{\nabla_1 f(\vx_k), \vv_{k}^1 - \vx_k^1}
    -  \frac{\sigma_{1}}{2}\|\vv_k^{1} - \vv_{k-1}^{1}\|^2,
\end{align*}
where we have used $U_k = f(\vy_k)$ (by the definition of $U_k$) and the equivalent expression for $A_k \Lambda_k - A_{k-1}\Lambda_{k-1}$ from  Eq.~\eqref{eq:lb-change-aarbcd}. As a practical matter, in even iterations, if $L_2$ is very large and potentially approaching $\infty,$ one can halt the backtracking line search and set $a_k = 0$ as soon as the search reaches some preset ``sufficiently small'' value of $a_k.$
%%%%%%%%%%%%%%%%%%%%%%%%
\section{Numerical Experiments}\label{sec:experiments}

\begin{figure*}[ht!]
\centering
\hspace{.3cm}
\subfigure[$N/n = 5$]{\includegraphics[width=0.22\textwidth]{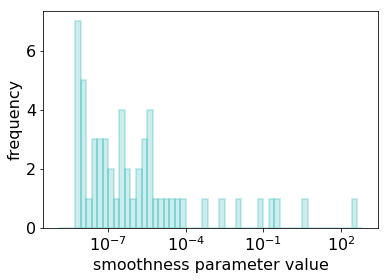}\label{fig:smooth-5-blog}}\hspace{.3cm}
\subfigure[$N/n = 10$]{\includegraphics[width=0.22\textwidth]{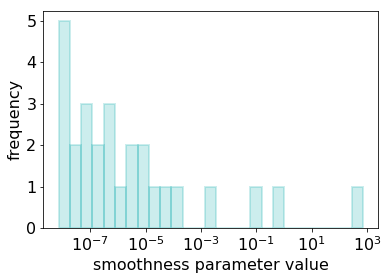}\label{fig:smooth-10-blog}}\hspace{.3cm}
\subfigure[$N/n = 20$]{\includegraphics[width=0.22\textwidth]{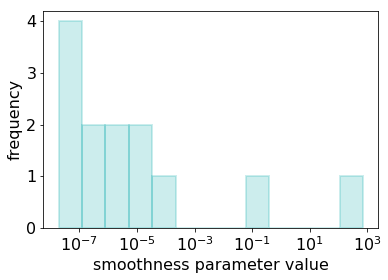}\label{fig:smooth-20-blog}}\hspace{.3cm}
\subfigure[$N/n = 40$]{\includegraphics[width=0.22\textwidth]{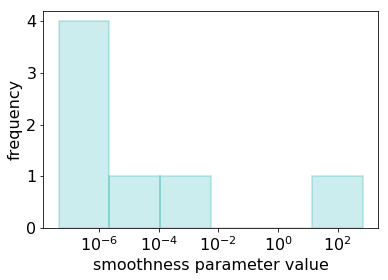}\label{fig:smooth-40-blog}}\hspace{-.1cm}
 \subfigure[$N/n = 5$]{\includegraphics[width=0.24\textwidth]{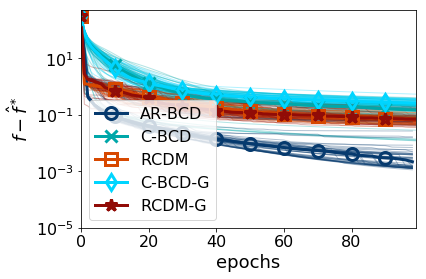}\label{fig:lip-5-blog}}
 \subfigure[$N/n=10$]{\includegraphics[width=0.24\textwidth]{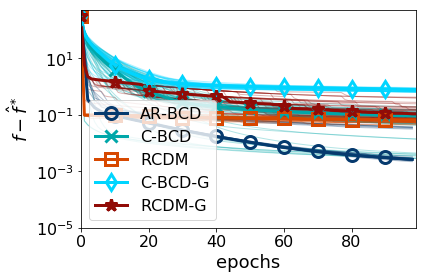}\label{fig:lip-10-blog}}
 \subfigure[$N/n=20$]{\includegraphics[width=0.24\textwidth]{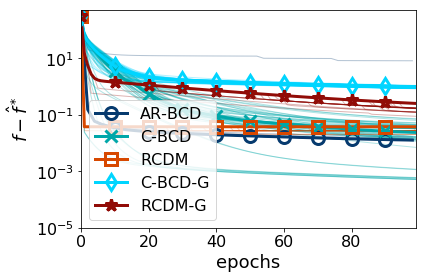}\label{fig:lip-med-20}}
 \subfigure[$N/n=40$]{\includegraphics[width=0.24\textwidth]{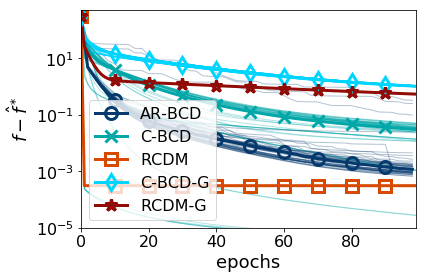}\label{fig:lip-med-40}}
\subfigure[$N/n = 5$]{\includegraphics[width=0.24\textwidth]{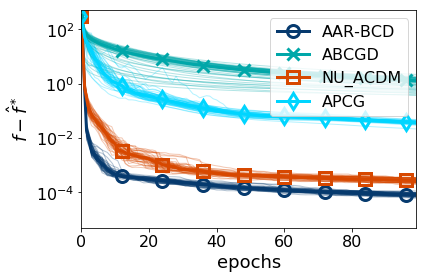}\label{fig:acc-lip-5-blog}}
 \subfigure[$N/n=10$]{\includegraphics[width=0.24\textwidth]{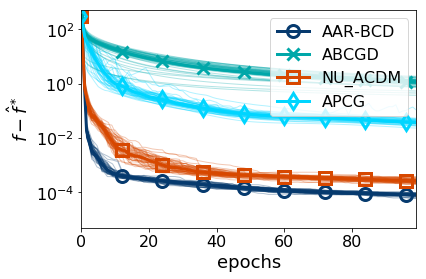}\label{fig:acc-lip-10-blog}}
 \subfigure[$N/n=20$]{\includegraphics[width=0.24\textwidth]{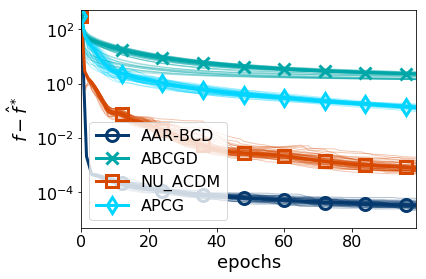}\label{fig:acc-lip-med-20}}
 \subfigure[$N/n=40$]{\includegraphics[width=0.24\textwidth]{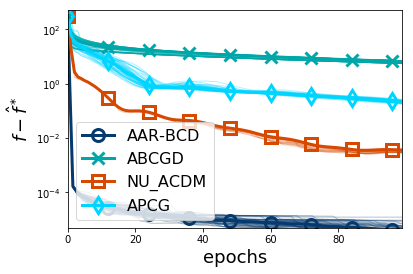}\label{fig:acc-lip-med-40}}
\caption{Comparison of different block coordinate descent methods: \protect\subref{fig:smooth-5-blog}-\protect\subref{fig:smooth-40-blog} distribution of smoothness parameters over blocks, \protect\subref{fig:lip-5-blog}-\protect\subref{fig:lip-med-40} comparison of non-accelerated methods, and \protect\subref{fig:acc-lip-5-blog}-\protect\subref{fig:acc-lip-med-40} comparison of accelerated methods. Block sizes $N/n$ increase going left to right.}
\label{fig:AR-BCD}
\end{figure*}

To illustrate the results, we solve the least squares problem on the BlogFeedback Data Set~\cite{buza2014feedback} obtained from UCI Machine Learning Repository~\cite{Lichman:2013}. The data set contains 280 attributes and 52,396 data points. The attributes correspond to various metrics of crawled blog posts. The data is labeled, and the labels correspond to the number of comments that were posted within 24 hours from a fixed basetime. The goal of a regression method is to predict the number of comments that a blog post receives. 

What makes linear regression with least squares on this dataset particularly suitable to our setting is that the smoothness parameters of individual coordinates in the least squares problem take values from a large interval, even when the data matrix $\mathbf{A}$ is scaled by its maximum absolute value (the values are between 0 and $\sim$354).\footnote{We did not compare \ref{eq:AR-BCD} and \ref{eq:AAR-BCD} to other methods on problems with a non-smooth block ($L_n = \infty$), as no other methods have any known theoretical guarantees in such a setting.} The minimum eigenvalue of $\mathbf{A}^T\mathbf{A}$ is zero (i.e., $\mathbf{A}^T\mathbf{A}$ is not a full-rank matrix), and thus the problem is not strongly convex. 

We partition the data into blocks as follows. We first sort the coordinates by their individual smoothness parameters. Then, we group the first $N/n$ coordinates (from the sorted list of coordinates) into the first block, the second $N/n$ coordinates into the second block, and so on. The chosen block sizes $N/n$ are 5, 10, 20, 40, corresponding to $n=\{56, 28, 14, 7\}$ coordinate blocks, respectively.

The distribution of the smoothness parameters over blocks, for all chosen block sizes, is shown in Fig.~\ref{fig:smooth-5-blog}-\ref{fig:smooth-40-blog}. Observe that as the block size increases (going from left to right in Fig.~\ref{fig:smooth-5-blog}-\ref{fig:smooth-40-blog}), the discrepancy between the two largest smoothness parameters increases. %Hence, we expect \ref{eq:AR-BCD} and \ref{eq:AAR-BCD} to exhibit higher improvements over other existing methods as the block size becomes larger. This is confirmed by our numerical experiments, as explained below.

In all the comparisons between the different methods, we define an epoch to be equal to $n$ iterations (this would correspond to a single iteration of a full-gradient method). The graphs plot the optimality gap of the methods over epochs, where the optimal objective value $f^*$ is estimated via a higher precision method and denoted by $\hat{f}^*$. All the results are shown for 50 method repetitions, with bold lines  representing the median\footnote{We choose to show the median as opposed to the mean, as it is well-known that in the presence of outliers the median is a robust estimator of the true mean~\cite{hampel2011robust}.} optimality gap over those 50 runs. The norm used in all the experiments is $\ell_2$, i.e., $\|\cdot\| = \|\cdot\|_2$.

\paragraph{Non-accelerated methods} We first compare \ref{eq:AR-BCD} with a gradient step to \ref{eq:standard-BCD}~\cite{nesterov2012efficiency} and standard cyclic BCD -- C-BCD (see, e.g.,~\cite{beck2013convergence-BCD}). To make the comparison fair, as \ref{eq:AR-BCD} makes two steps per iteration, we slow it down by a factor of two compared to the other methods (i.e., we count one iteration of \ref{eq:AR-BCD} as two). In the comparison, we consider two cases for RCDM and C-BCD: (i) the case in which these two algorithms perform gradient steps on the first $n-1$ blocks and exact minimization on the $n^{\mathrm{th}}$ block (denoted by RCDM and C-BCD in the figure), and (ii) the case in which the algorithms perform gradient steps on all blocks (denoted by RCDM-G and C-BCD-G in the figure). The sampling probabilities for RCDM and AR-BCD are proportional to the block smoothness parameters. The permutation for C-BCD is random, but fixed in each method run.

Fig.~\ref{fig:lip-5-blog}-\ref{fig:lip-med-40} shows the comparison of the described non-accelerated algorithms, for block sizes $N/n \in\{5, 10, 20, 40\}$. The first observation to make is that adding exact minimization over the least smooth block speeds up the convergence of both C-BCD and RCDM, suggesting that the existing analysis of these two methods is not tight. Second, \ref{eq:AR-BCD} generally converges to a lower optimality gap. While RCDM makes a large initial progress, it stagnates afterwards due to the highly non-uniform sampling probabilities, whereas \ref{eq:AR-BCD} keeps making progress.

\paragraph{Accelerated methods} 
Finally, we compare \ref{eq:AAR-BCD} to NU\_ACDM~\cite{allen2016even}, APCG~\cite{lin2014accelerated}, and accelerated C-BCD (ABCGD) from~\cite{beck2013convergence-BCD}. As \ref{eq:AAR-BCD} makes three steps per iteration (as opposed to two steps normally taken by other methods), we slow it down by a factor 1.5 (i.e., we count one iteration of \ref{eq:AAR-BCD} as 1.5). We chose the sampling probabilities of NU\_ACDM and \ref{eq:AAR-BCD} to be proportional to $\sqrt{L_i}$, while the sampling probabilities for APCG are uniform\footnote{The theoretical results for APCG were only presented for uniform sampling~\cite{lin2014accelerated}.}. Similar as before, each full run of ABCGD is performed on a random but fixed permutation of the blocks.

The results are shown in Fig.~\ref{fig:acc-lip-5-blog}-\ref{fig:acc-lip-med-40}. Compared to APCG (and ABCGD), NU\_ACDM and \ref{eq:AAR-BCD} converge much faster, which is expected, as the distribution of the smoothness parameters is highly non-uniform and the meethods with non-uniform sampling are theoretically faster by factor of the order $\sqrt{n}$~\cite{allen2016even}. As the block size is increased (going left to right), the discrepancy between the smoothness parameters of the least smooth block and the remaining blocks increases, and, as expected, \ref{eq:AAR-BCD} exhibits more dramatic improvements compared to the other methods. 

\section{Conclusion}\label{sec:conclusion}

We presented a novel block coordinate descent algorithm \ref{eq:AR-BCD} and its accelerated version for smooth minimization \ref{eq:AAR-BCD}. Our work answers the open question of~\cite{beck2013convergence-BCD} whether the convergence of block coordinate descent methods intrinsically depends on the largest smoothness parameter over all the blocks by showing that such a dependence is not necessary, as long as exact minimization over the least smooth block is possible. Before our work, such a result only existed for the setting of two blocks, using the alternating minimization method.

There are several research directions that merit further investigation. For example, we observed empirically that exact optimization over the non-smooth block improves the performance of RCDM and C-BCD, which is not justified by the existing analytical bounds. We expect that in both of these methods the dependence on the least smooth block can be removed, possibly at the cost of a worse dependence on the number of blocks. Further, \ref{eq:AR-BCD} and \ref{eq:AAR-BCD} are mainly useful when the discrepancy between the largest block smoothness parameter and the remaining smoothness parameters is large, while under uniform distribution of the smoothness parameters it can be slower than other methods by a factor 1.5-2. It is an  interesting question whether there are modifications to \ref{eq:AR-BCD} and \ref{eq:AAR-BCD} that would make them uniformly better than the alternatives.

\section{Acknowledgements}
Part of this work was done while the authors were visiting the Simons Institute for the Theory of Computing. It was partially supported by NSF grant \#CCF-1718342 and by the DIMACS/Simons Collaboration on Bridging Continuous and Discrete Optimization through NSF grant \#CCF-1740425.
%\clearpage
\balance
\bibliography{references_BCD}
\bibliographystyle{icml2018}
%\clearpage
\appendix
\onecolumn
\section{Omitted Proofs from Section~\ref{sec:acc-algo}}\label{app:omitted-proofs}

\begin{proof}[Proof of Proposition~\ref{prop:expect-Delta_k}]
Let $\mathcal{F}_{k-1}$ be the natural filtration up to iteration $k-1$. Observe that, as $\nabla_n f(\vx_k) = \zeros$:
\begin{equation}\label{eq:expectation-of-Delta_k}
\mathbb{E}[\Delta_k|\mathcal{F}_{k-1}] = \nabla f(\vx_k).
\end{equation}

Since $\vx_1$ is deterministic (fixed initial point) and the only random variable $\Delta_1$  depends on is $i_1$, we have:
\begin{equation}\label{eq:initial-Delta}
\begin{aligned}
\mathbb{E}[a_1\innp{\Delta_1, \vx_* - \vx_1}] &= a_1\innp{\nabla f(\vx_1), \vx_* - \vx_1}\\
&= \mathbb{E}[a_1\innp{\nabla f(\vx_1), \vx_* - \vx_1}].
\end{aligned}
\end{equation}
Let $k > 1$. Observe that $a_j\innp{\Delta_j, \vx_* - \vx_j}$ is measurable with respect to $\mathcal{F}_{k-1}$ for $j \leq k-1$. By linearity of expectation, using (\ref{eq:expectation-of-Delta_k}):
\begin{align}
\mathbb{E}[\littlesum_{j=1}^k a_j \innp{\Delta_j, \vx_* - \vx_j}|\mathcal{F}_{k-1}]%\notag\\
 = a_k \innp{\nabla f(\vx_k), \vx_* - \vx_k} + \littlesum_{j=1}^{k-1}a_j  \innp{\Delta_j, \vx_* - \vx_j}.\notag
\end{align}
Taking expectations on both sides of the last equality gives a recursion on $\mathbb{E}[\littlesum_{j=1}^k a_j \innp{\Delta_j, \vx_* - \vx_j}]$, which, combined with (\ref{eq:initial-Delta}), completes the proof.
\end{proof}

\begin{proof}[Proof of Lemma~\ref{lemma:AkGammak-supermartingale}]
As $A_{k-1}\Gamma_{k-1}$ is measurable with respect to the natural filtration $\mathcal{F}_{k-1}$, $\mathbb{E}[A_k\Gamma_k|\mathcal{F}_{k-1}] \leq A_{k-1}\Gamma_{k-1}$ is equivalent to $\mathbb{E}[A_k\Gamma_k - A_{k-1}\Gamma_{k-1}|\mathcal{F}_{k-1}] \leq 0$. 

The change in the upper bound is:
\begin{align*}
 A_k U_k - A_{k-1}U_{k-1}
= A_{k}(f(\vy_k) - f(\vx_k)) + A_{k-1}(f(\vx_k) - f(\vy_{k-1}))+ a_k f(\vx_k).
\end{align*}
By convexity, $f(\vx_k) - f(\vy_{k-1})\leq \innp{\nabla f(\vx_k), \vx_k - \vy_{k-1}}$. Further, as $\vy_k = \vx_k + I_N^{i_k}\frac{a_k}{p_{i_k}A_k}(\vv_k - \vv_{k-1})$, we have, by smoothness of $f(\cdot)$, that $f(\vy_k) - f(\vx_k) \leq \innp{\nabla f(\vx_k), I_N^{i_k}\frac{a_k}{p_{i_k}A_k}(\vv_k - \vv_{k-1})} + \frac{L_{i_k}{a_k}^2}{2{p_{i_k}}^2 {A_k}^2}\|\vv_k^{i_k} - \vv_{k-1}^{i_k}\|^2$. Hence:
\begin{equation}\label{eq:ub-change-aarbcd}
\begin{aligned}
& A_k U_k - A_{k-1}U_{k-1} \\
&\hspace{1cm}\leq a_k f(\vx_k ) + \innp{\nabla f(\vx_k), A_{k-1}(\vx_k - \vy_{k-1}) + I_N^{i_k}\frac{a_k}{p_{i_k}}(\vv_k - \vv_{k-1})}
+ \frac{L_{i_k}{a_k}^2}{2{p_{i_k}}^2 {A_k}}\|\vv_k^{i_k} - \vv_{k-1}^{i_k}\|^2.
\end{aligned}
\end{equation}

% By convexity of $f(\cdot)$, $f(\vx_k) - f(\vy_{k-1})\leq \innp{\nabla f(\vx_k), \vx_k - \vy_{k-1}}$, while by the definition of the steps (\ref{eq:AAR-BCD}) and smoothness of $f(\cdot)$, $f(\vy_k) - f(\vx_k) \leq \innp{\nabla_i f(\vx_k), \vy_k^i - \vx_k^i} + \frac{L_i}{2}\|\vy_k^i - \vx_k^i\|^2$. Hence:
% \begin{align*}
% &\mathbb{E}[A_k U_k - A_{k-1}U_{k-1}|\mathcal{F}_{k-1}]\\
% &\leq a_k f(\vx_k) + \littlesum_{i=1}^{n-1} p_i \frac{A_k L_i}{2} \|\vy_k^i - \vx_k^i\|^2\\
% &+ \littlesum_{i=1}^{n-1} \innp{\nabla_i f(\vx_k), p_i A_k (\vy_k^i - \vx_k^i) + A_{k-1}(\vx_k^i -\vy_{k-1}^i)}.
% \end{align*}
Let $m_k(\vu) = \sum_{j=1}^k a_j\innp{\Delta_j, \vu - \vx_j} +  \littlesum_{i=1}^n \frac{\sigma_i}{2}\|\vu^i - \vx_1^i\|^2$ denote the function under the minimum in the definition of $\Lambda_k$. Observe that $m_k(\vu) = m_{k-1}(\vu) + a_k \innp{\Delta_k, \vu - \vx_k}$ and $\vv_k = \argmin_{\vu}m_k(\vu)$. Then:
\begin{align*}
m_{k-1}(\vv_k) =& m_{k-1}(\vv_{k-1}) + \innp{\nabla m_{k-1}(\vv_{k-1}), \vv_k - \vv_{k-1}}
+ \littlesum_{i=1}^{n-1} \frac{\sigma_i}{2}\|\vv_k^i - \vv_{k-1}^i\|^2\\
=& m_{k-1}(\vv_{k-1}) +  \frac{\sigma_{i_k}}{2}\|\vv_k^{i_k} - \vv_{k-1}^{i_k}\|^2,
\end{align*}
as $\vv_k$ and $\vv_{k-1}$ only differ over the block $i_k$ and $\vv_{k-1} = \argmin_{\vu}m_{k-1}(\vu)$ (and, thus, $\nabla m_{k-1}(\vv_{k-1}) = \zeros$).

%By duality of the norms (see, e.g., Exercise 3.27 in \cite{boyd2004convex}):
%\begin{equation*}
%\frac{\sigma_{i_k}}{2}\|\vv_k^{i_k} - \vv_{k-1}^{i_k}\|^2 \geq a_k\innp{\Delta_k, \vv_{k-1} - \vv_{k}} - \frac{{a_k}^2}{2\sigma_{i_k}}\|\Delta_k\|_*^2.
%\end{equation*}
Hence, it follows that $m_k(\vv_k) - m_{k-1}(\vv_{k-1}) =  a_k\innp{\Delta_k, \vv_{k} - \vx_k} +  \frac{\sigma_{i_k}}{2}\|\vv_k^{i_k} - \vv_{k-1}^{i_k}\|^2$, and, thus:
\begin{equation}\label{eq:lb-change-aarbcd}
\begin{aligned}
A_k\Lambda_k - A_{k-1}\Lambda_{k-1} = a_k f(\vx_k) +  a_k\innp{\Delta_k, \vv_{k} - \vx_k} +  \frac{\sigma_{i_k}}{2}\|\vv_k^{i_k} - \vv_{k-1}^{i_k}\|^2.
\end{aligned}
\end{equation}

Combining (\ref{eq:ub-change-aarbcd}) and (\ref{eq:lb-change-aarbcd}):
\begin{align*}
A_k \Gamma_k - A_{k-1}\Gamma_{k-1}
\leq & \innp{\nabla f(\vx_{k}), A_{k-1}(\vx_{k}-\vy_{k-1}) + I_N^{i_k}\frac{a_k}{p_{i_k}}(\vv_k - \vv_{k-1}) } -  a_k\innp{\Delta_k, \vv_{k} - \vx_k}\\
& + \frac{L_{i_k}{a_k}^2}{2{p_{i_k}}^2A_k }\|\vv_k^{i_k} - \vv_{k-1}^{i_k}\|^2 - \frac{\sigma_{i_k}}{2}\|\vv_k^{i_k} - \vv_{k-1}^{i_k}\|^2\\
&\leq \innp{\nabla f(\vx_{k}), A_{k-1}(\vx_{k}-\vy_{k-1}) + I_N^{i_k}\frac{a_k}{p_{i_k}}(\vv_k - \vv_{k-1})} -  a_k\innp{\Delta_k, \vv_{k} - \vx_k},
\end{align*}
as, by the initial assumptions, $\frac{{a_k}^2}{A_k}\leq \frac{p_{i_k}^2 \sigma_{i_k}}{L_{i_k}}$. % and $\vy_k = \vx_k + I_N^{i_k}\frac{a_k}{p_{i_k}A_k}(\vv_k - \vv_{k-1})$. 
%

%Recall that, by definition, $\Delta_k = \frac{\nabla_{i_k}f(\vx_k)}{p_{i_k}}$. 
Finally, taking expectations on both sides, and as $\vx_k, \vy_{k-1}, \vv_{k-1}$ are all measurable w.r.t. $\mathcal{F}_{k-1}$ and by the separability of the terms in the definition of $\vv_k$:
\begin{align*}
\mathbb{E}[A_k \Gamma_k - A_{k-1}\Gamma_{k-1}|\mathcal{F}_{k-1}]
 \leq \innp{\nabla f(\vx_k), A_k \vx_k - A_{k-1}\vy_{k-1} - a_k \vv_{k-1}}=0,
\end{align*}
as, from (\ref{eq:AAR-BCD}), $\vx_k = \vxh_k = \frac{A_{k-1}}{A_k}\vy_{k-1} + \frac{a_k}{A_k}\vv_{k-1}$ over all the blocks except for the block $n,$ while $\nabla_n f(\vx_k) = \zeros,$ as $\vx_k$ is the minimizer of $f$ over block $n,$ when other blocks in $\vxh_k$ are fixed.\footnote{Previous version of the proof provided an incomplete justification for the last expression in the proof being equal to zero; we thank Sergey Guminov for pointing this out.} 
\end{proof}

\section{Efficient Implementation of \ref{eq:AAR-BCD} Iterations}\label{app:efficient-iterations}

Using similar ideas as in \cite{fercoq2015accelerated,lin2014accelerated,lee2013efficient}, here we discuss how to efficiently implement iterations of \ref{eq:AAR-BCD}, without requiring full-vector updates. 
First, due to the separability of the terms inside the minimum, between successive iterations $\vv_k$ changes only over a single block. This is formalized in the following simple proposition.
\begin{proposition}\label{prop:changes-over-a-block}
In each iteration $k \geq 1,$ $\vv_k^{i} = \vv_{k-1}^{i}$, $\forall i \neq i_k$ and $\vv_k^{i_k} = \vv_{k-1}^{i_k} + \vw^{i_k}$, where:
$$
\vw^{i_k} = \argmin_{\vu^{i_k}}\{a_k \innp{\Delta_k^{i_k}, \vu} + \frac{\sigma_{i_k}}{2}\|\vu^{i_k} - \vv_{k-1}^{i_k}\|^2\}.
$$
\end{proposition}
\begin{proof}
Recall the definition of $\vv_k$. We have:
\begin{align*}
\vv_k =& \argmin_{\vu}\Big\{\sum_{j=1}^{k} \innp{\Delta_j, \vu} + \sum_{i=1}^{n-1} \frac{\sigma_i}{2}\|\vu^{i} - \vx_1^{i}\|^2\Big\}\\
=& \argmin_{\vu}\Big\{ \sum_{j=1}^{k-1} \innp{\Delta_j, \vu} + \innp{\Delta_k, \vu} + \sum_{i=1}^{n-1} \frac{\sigma_i}{2}\|\vu^{i} - \vx_1^{i}\|^2 \Big\}\\
=& \argmin_{\vu}\Big\{ \sum_{j=1}^{k-1} \innp{\Delta_j, \vu} + \innp{\Delta_k^{i_k}, \vu^{i_k}} + \sum_{i=1}^{n-1} \frac{\sigma_i}{2}\|\vu^{i} - \vx_1^{i}\|^2 \Big\}\\
=& \vv_{k-1} + \argmin_{\vu^{i_k}}\Big\{ \innp{\Delta_k^{i_k}, \vu^{i_k}} + \frac{\sigma_{i_k}}{2}\|\vu^{i_k} - \vv_{k-1}^{i_k}\|^2 \Big\},
\end{align*}
where the third equality is by the definition of $\Delta_k$ ($\Delta_k^i = 0$ for $i \neq i_k$) and the last equality follows from block-separability of the terms under the min.
\end{proof}

Since $\vv_k$ only changes over a single block, this will imply that the changes in $\vx_k$ and $\vy_k$ can be localized. In particular, let us observe the patterns in changes between successive iterations. We have that, $\forall i \neq n:$
\begin{equation}\label{eq:change-in-x-i}
\begin{aligned}
\vx_k^{i} &= \frac{A_{k-1}}{A_k} \vy_{k-1}^i + \frac{a_k}{A_k}\vv_{k-1}^i= \frac{A_{k-1}}{A_k}\left(\vy_{k-1}^i - \vv_{k-1}^i\right) + \vv_{k-1}^i
\end{aligned}
\end{equation}
and
\begin{equation}\label{eq:change-in-y-i}
\begin{aligned}
\vy_k^i &= \vx_k^i + \frac{1}{p_i}\frac{a_k}{A_k}\left(\vv_k^i - \vv_{k-1}^i\right)\\
&= \frac{A_{k-1}}{A_k}\left(\vy_{k-1}^i - \vv_{k-1}^i\right) + \left(1 - \frac{1}{p_i}\frac{a_k}{A_k}\right)\left(\vv_{k-1}^i - \vv_k^i\right) + \vv_k^i.
\end{aligned}
\end{equation}
Due to Proposition~\ref{prop:changes-over-a-block}, $\vv_k$ and $\vv_{k-1}$ can be computed without full-vector operations (assuming the gradients can be computed without full-vector operations, which we will show later in this section). Hence, we need to show that it is possible to replace $\frac{A_{k-1}}{A_k}\left(\vy_{k-1}^i - \vv_{k-1}^i\right)$ with a quantity that can be computed without the full-vector operations. Observe that $\vy_0 - \vv_0 = 0$ (from the initialization of~\eqref{eq:AAR-BCD}) and that, from~\eqref{eq:change-in-y-i}:
\begin{equation}
\vy_k^i - \vv_k^i = \frac{A_{k-1}}{A_k}\left(\vy_{k-1}^i - \vv_{k-1}^i\right) + \left(1 - \frac{1}{p_i}\frac{a_k}{A_k}\right)\left(\vv_{k-1}^i - \vv_k^i\right).\notag
\end{equation}
Dividing both sides by $\frac{{a_k}^2}{{A_{k}}^2}$ and assuming that $\frac{{a_k}^2}{A_k}$ is constant over iterations, we get:
\begin{equation}\label{eq:recursion-y-v}
\frac{{A_k}^2}{{a_k}^2}\left(\vy_k^i - \vv_k^i\right) = \frac{{A_{k-1}}^2}{{a_{k-1}}^2}\left(\vy_{k-1}^i - \vv_{k-1}^i\right) + \frac{{A_k}^2}{{a_k}^2}\left(1 - \frac{1}{p_i}\frac{a_k}{A_k}\right)\left(\vv_{k-1}^i - \vv_k^i\right).
\end{equation}
Let $N_n$ denote the size of the $n^{\mathrm{th}}$ block and define the $(N-N_n)$-length vector $\vu_k$ by $\vu_k^i = \frac{{A_k}^2}{{a_k}^2}\left(\vy_k^i - \vv_k^i\right)$, $\forall i \neq n$. Then (from~\eqref{eq:recursion-y-v}) $\vu_k^i = \vu_{k-1}^{i}  + \frac{{A_k}^2}{{a_k}^2}\left(1 - \frac{1}{p_i}\frac{a_k}{A_k}\right)\left(\vv_{k-1}^i - \vv_k^i\right),$ and, hence, in iteration $k$, $\vu_k$ changes only over block $i_k$. Combining with~\eqref{eq:change-in-x-i} and~\eqref{eq:change-in-y-i}, we have the following lemma.
\begin{lemma}\label{lemma:efficient-comp-of-y-x}
Assume that $\frac{{a_k}^2}{A_{k}}$ is kept constant over the iterations of~\ref{eq:AAR-BCD}. Let $\vu_k$ be the $(N-N_n)$-dimensional vector defined recursively as $\vu_0 = \zeros$, $\vu_k^{i} = \vu_{k-1}^i$ for $i \in \{1, ..., n-1\}$, $i \neq i_k$ and $\vu_k^{i_k} = \vu_{k-1}^{i_k} + \frac{{A_k}^2}{{a_k}^2}\left(1 - \frac{1}{p_i}\frac{a_k}{A_k}\right)\left(\vv_{k-1}^i - \vv_k^i\right)$. Then, $\forall i \in \{1,..., n-1\}$: $\vx_k^i = \frac{{a_k}^2}{{A_k}^2}\vu_{k-1}^i + \vv_{k-1}^i$ and $\vy_k^i = \frac{{a_k}^2}{{A_k}^2}\vu_{k-1}^i + \left(1 - \frac{1}{p_i}\frac{a_k}{A_k}\right)\left(\vv_{k-1}^i - \vv_k^i\right) + \vv_k^i$.
\end{lemma}
Note that we will never need to explicitly compute $\vx_k, \vy_k$, except for the last iteration $K$, which outputs $\vy_K$. To formalize this claim, we need to show that we can compute the gradients $\nabla_i f(\vx_k)$ without explicitly computing $\vx_k$ and that we can efficiently perform the exact minimization over the $n^{\mathrm{th}}$ block. This will only be possible by assuming specific structure of the objective function, as is typical for accelerated block-coordinate descent methods~\cite{fercoq2015accelerated,lee2013efficient,lin2014accelerated}. In particular, we assume that for some $m \times N$ dimensional matrix $\mM:$
\begin{equation}\label{eq:objective-structure}
f(\vx) = \sum_{j=1}^m \phi_j(e_j^T\mM\vx) + \psi(\vx),
\end{equation}
where $\phi_j: \mathbb{R}\rightarrow \mathbb{R}$ and $\psi = \sum_{i=1}^n \psi_i:\mathbb{R}^N\rightarrow \mathbb{R}$ is block-separable. 
\paragraph{Efficient Gradient Computations.} Assume for now that $\vx_k^n$ can be computed efficiently (we will address this at the end of this section). Let ${ind}$ denote the set of indices of the coordinates from blocks $\{1, 2, ..., n-1\}$ and denote by $\mB$ the matrix obtained by selecting the columns of $\mM$ that are indexed by ${ind}$. Similarly, let $ind_n$ denote the set of indices of the coordinates from block $n$ and let $\mC$ denote the submatrix of $\mM$ obtained by selecting the columns of $\mM$ that are indexed by $ind_n$.  Denote $\vr_{\vu_k} = \mB \vu_k$, $\vr_{\vv_k} = \mB [\vv_k^{1}, \vv_k^2, ... , \vv_k^{n-1}]^T$, $\vr_n = \mC \vx_k^n$. Let $ind_{i_k}$ be the set of indices corresponding to the coordinates from block $i_k$. Then:
\begin{equation}\label{eq:grad-computation}
\nabla_{i_k} f(\vx_k) = \sum_{j=1}^m (\mM_{j, ind_{ik}})^T \phi_j'\left(\frac{{a_k}^2}{{A_k}^2}\vr_{\vu_{k-1}}^j + \vr_{\vv_{k-1}}^j + \vr_n^j\right) + \nabla_{i_k}\psi(\vx).
\end{equation}
Hence, as long as we maintain $\vr_{\vu_k}, \vr_{\vv_k},$ and $\vr_n$ (which do not require full-vector operations), we can efficiently compute the partial gradients $\nabla_{i_k}f(\vx_k)$ without ever needing to perform any full-vector operations.
\paragraph{Efficient Exact Minimization.} Suppose first that $\psi(\vx)\equiv 0$. Then:
$$
\vr_n = \argmin_{\vr \in \mathbb{R}^{m}}\left\{\sum_{j=1}^m \phi_j\left(\frac{{a_k}^2}{{A_k}^2}\vr_{\vu_{k-1}}^j + \vr_{\vv_{k-1}}^j + \vr^j\right)\right\},
$$
and $\vr_n$ can be computed but solving $m$ single-variable minimization problems, which can be done in closed form or with a very low complexity. Computing $\vr_n$ is sufficient for defining all algorithm iterations, except for the last one (that outputs a solution). Hence, we only need to compute $\vx_k^n$ once -- in the last iteration. 

More generally, $\vx_k^n$ is determined by solving:
$$
\vx_k^n = \argmin_{\vx \in \mathbb{R}^{N_n}}\left\{\sum_{j=1}^m \phi_j\left(\frac{{a_k}^2}{{A_k}^2}\vr_{\vu_{k-1}}^j + \vr_{\vv_{k-1}}^j + (\mC\vx)^j\right) + \psi_n(\vx)\right\}.
$$
When $m$ and $N_n$ are small, high-accuracy polynomial-time convex optimization algorithms are computationally inexpensive, and $\vx_k^n$ can be computed efficiently. 

In the special case of linear and ridge regression, $\vx_k^n$ can be computed in closed form, with minor preprocessing. In particular, if $\vb$ is the vector of labels, then the problem becomes:
$$
\vx_k^n = \argmin_{\vx \in \mathbb{R}^{N_n}}\left\{\sum_{j=1}^m \left(\frac{{a_k}^2}{{A_k}^2}\vr_{\vu_{k-1}}^j + \vr_{\vv_{k-1}}^j + (\mC\vx)^j - \vb^j\right)^2 + \frac{\lambda}{2}\|\vx\|_2^2\right\},
$$
where $\lambda = 0$ in the case of (simple) linear regression. Let $\vb' = \vb - \frac{{a_k}^2}{{A_k}^2}\vr_{\vu_{k-1}} - \vr_{\vv_{k-1}}$. Then:
$$
\vx_k^n = (\mC^T\mC + \lambda \mathbf{I})^{\dagger} (\mC^T \vb'),
$$
where $(\cdot)^{\dagger}$ denotes the matrix pseudoinverse, and $\mathbf{I}$ is the identity matrix. Since $\mC^T\mC + \lambda \mathbf{I}$ does not change over iterations, $(\mC^T\mC + \lambda \mathbf{I})^{\dagger}$ can be computed only once at the initialization. Recall that $\mC^T\mC + \lambda \mathbf{I}$ is an $N_n \times N_n$ matrix, where $N_n$ is the size of the $n^{\mathrm{th}}$ block, and thus inverting $\mC^T\mC + \lambda \mathbf{I}$ is computationally inexpensive as long as $N_n$ is not too large. This reduces the overall per-iteration cost of the exact minimization to about the same cost as for performing gradient steps.
% $$
% m_{k-1}(\vv_k) \geq m_{k-1}(\vv_{k-1}) + D_{\psi}(\vv_k, \vv_{k-1}).
% $$
% As $\vv_k$ and $\vv_{k-1}$ differ only over the block $i$:
% \begin{align*}
% &m_k(\vv_k) - m_{k-1}(\vv_{k-1}) = a_k \innp{\Delta_k, \vv_k - \vx_k}\\
% &\hspace{1cm}+ D_{\psi_i}(\vv_k^i, \vv_{k-1}^i).
% \end{align*}
% Thus, the change in the lower bound $\Lambda_k$ is:
% \begin{align*}
% &A_k \Lambda_k  - A_{k-1}\Lambda_{k-1} = a_k f(\vx_k) + a_k \innp{\Delta_k, \vv_k - \vx_k}\\
% &\hspace{1cm}+ D_{\psi_i}(\vv_k^i, \vv_{k-1}^i).
% \end{align*}
% As $\psi_i$ is $\sigma_i$-strongly convex, $D_{\psi_i}(\vv_k^i, \vv_{k-1}^i) \geq \frac{\sigma_i}{2}\|\vv_k^i - \vv_{k-1}^i\|^2$. Taking the expectation of $A_k \Lambda_k  - A_{k-1}\Lambda_{k-1}:$
% \begin{align*}
% \mathbb{E}[& A_k \Lambda_k  - A_{k-1}\Lambda_{k-1}|\mathcal{F}_{k-1}] \geq  a_k f(\vx_k)\\
% &+ a_k \innp{\nabla f(\vx_k), \vv_k - \vx_k} + \littlesum_{i=1}^{n-1}\frac{p_i\sigma_i}{2}\|\vv_k^i - \vv_{k-1}^i\|^2.  %+\frac{a_k \mu}{2}\|\vv_k - \vx_k\|^2.
% \end{align*}
% Combining the bounds on the change in the upper and lower bounds and using that $\vy_k^i - \vx_k^i = \frac{a_k}{p_i A_k}(\vv_k^i - \vv_{k-1}^{i})$:
% \begin{align*}
% &\mathbb{E}[A_k\Gamma_k - A_{k-1}\Gamma_{k-1}|\mathcal{F}_{k-1}]\\
% & \leq \littlesum_{i=1}^{n-1}\Big(\frac{{a_k}^2L_i}{2p_iA_k} - \frac{p_i\sigma_i}{2}\Big) \|\vv_k^i - \vv_{k-1}^i\|^2 \leq 0,
% \end{align*}
% as, by the lemma's assumptions, $\frac{{a_k}^2}{A_k}\leq \frac{{p_i}^2\sigma_i}{L_i}$, $\forall i$.
%\end{proof}

\end{document}